\documentclass[a4paper,12pt]{amsart}
\usepackage{amsfonts,amsthm,amsmath,amssymb,graphicx}

\title[increasing digits]{Increasing digit subsystems of infinite iterated function systems}

\date{}
\author{Thomas Jordan}
\address{Thomas Jordan\\Department of Mathematics\\ The University of Bristol\\
University Walk\\Clifton\\ Bristol\\BS8 1TW\\UK}
\email{thomas.jordan@bristol.ac.uk}
\author{Micha\l\ Rams}
\address{Micha\l\ Rams\\Institute of Mathematics\\ Polish Academy of Sciences\\ ul.
\'Sniadeckich 8, 00-956 Warszawa\\ Poland }
\email{M.Rams@impan.gov.pl}
\thanks{The research of M.R. was supported by grants EU FP6 ToK SPADE2, EU FP6 RTN
CODY and MNiSW grant 'Chaos, fraktale i dynamika konforemna'. }

\theoremstyle{plain}
\newtheorem{lem}{Lemma}[section]

\newtheorem{thm}[lem]{Theorem}
\newtheorem{cor}[lem]{Corollary}

\theoremstyle{definition}

\theoremstyle{remark}

\numberwithin{equation}{section}

\newcommand{\R}{\mathbb R}

\newcommand{\N}{\mathbb N}

\renewcommand{\epsilon}{\varepsilon}
\newcommand{\dimp}{\dim_{P}}
\newcommand{\dimh}{\dim_{H}}
\begin{document}
\begin{abstract}
We consider an infinite iterated function system
$\{f_i\}_{i=1}^{\infty}$ on $[0,1]$ with a polynomially increasing
contraction rate. We look at subsets of such systems where we only
allow iterates $f_{i_1}\circ f_{i_2}\circ f_{i_3}\circ\cdots$ if
$i_n>\Phi(i_{n-1})$ for certain increasing functions
$\Phi:\N\rightarrow\N$. We compute both the Hausdorff and packing
dimensions of such sets. Our results generalise work of Ramharter which shows that the set of continued fractions with
strictly increasing digits has Hausdorff dimension $\frac{1}{2}$.
\end{abstract}
\maketitle

\def\thefootnote{}
\footnote{2010{\it Mathematics
Subject Classification}: Primary 28A80, Secondary 11K50}
\def\thefootnote{\arabic{footnote}}
\setcounter{footnote}{0}

\section{Introduction}
In this paper we consider certain subsets of the attractors of infinite iterated function systems. 
For each $n\in\N$ we will let $f_n:[0,1]\rightarrow [0,1]$ be
$C^1$ maps such that
\begin{enumerate}
\item\label{one} There exists $m\in\N$ and $0<A<1$ such that for
all $(a_1,\ldots,a_m)\in\N^m$ and for all $x\in [0,1]$
$$0<|(f_{a_1}\circ\cdots\circ f_{a_m})'(x)|\leq A<1.$$
\item\label{two} For any $i,j\in\N$ $f_i((0,1))\cap
f_j((0,1))=\emptyset$. \item\label{three} There exist $d>1$ such
that for any $\varepsilon>0$ there exist
$C_1(\varepsilon),C_2(\varepsilon)>0$ such that for $i\in\N$ there
exist constants $\lambda_i, \xi_i$ such that for all $x\in [0,1]$
$\xi_i\leq |f_i'(x)|\leq\lambda_i$ and
$$\frac{C_1}{i^{d+\varepsilon}}\leq\xi_i\leq\lambda_i\leq\frac{C_2}{i^{d-\varepsilon}}$$
\end{enumerate}
We will call such a system a $d$-decaying system.

There will be a natural projection $\Pi:\N^{\N}\rightarrow [0,1]$
defined by
$$\Pi(\underline{a})=\lim_{n\rightarrow\infty}f_{a_1}\circ\cdots\circ f_{a_n}(1).$$
We will denote $\Lambda=\Pi(\N^{\N})$ as the attractor of the
system. We will let $T:\Lambda\rightarrow \Lambda$ be the
expanding map defined by $T(x)=f_i^{-1}(x)$ if $x\in f_i([0,1])$.
If $x=\Pi(\underline{a})$ then we will refer to $\{a_n\}_{n\in\N}$
as the digits of $x$ (these are not necessarily unique). For
brevity of notation for $x\in\Lambda$, $\{a_i(x)\}_{i\in\N}$ will
denote a sequence $\underline{a}\in\N^{\N}$ such that
$\Pi(\underline{a})=x$. 

We are interested in the set of $x$ where
the digits are increasing monotonically. For an function
$\Phi:\N\rightarrow\R$ which satisfies that $\Phi(n)\geq n$ we
will denote
\begin{equation}\label{def1}
X_{\Phi}=\Pi\{\underline{a}:a_{n+1}>\Phi(a_n)\text{ for all }n\in\N\}.
\end{equation}
We will be looking at what the dimension of these sets for various
different notions of dimension. We will be considering Hausdorff
dimension, denoted $\dimh$, packing dimension, denoted $\dimp$ and
upper box counting dimension denoted $\overline{\dim}_B$. For the
definitions of these notions of dimension the reader is referred
to \cite{Falcbook}. Our first result is the following
\begin{thm} \label{thm:main}
Let $\Phi:\N\rightarrow\R$ satisfy
that for some $\beta\geq 1$ we have $n\leq\Phi(n)\leq\beta n$ for all
$n\in N$. We then have that
$$\dimh X_{\Phi}=\frac{1}{d}.$$
\end{thm}
Considering packing dimension instead of Hausdorff dimension we
obtain the following, stronger result:
\begin{thm}\label{packing}
Let $s_0=\overline{\dim}_B(\{f_i(0)\}_{i=1}^{\infty})$ and
$\Phi:\N\rightarrow\R$ satisfy that $\Phi(n)\geq n$ then we have
that
$$\dimp X_{\Phi}=\max\left\{s_0,\frac{1}{d}\right\}.$$
\end{thm}

To look at $\dim_HX_{\Phi}$ for functions $\Phi:\N\rightarrow\R$ where the growth rate is quicker than a linear rate we restrict ourselves to a certain class of $d$-decaying systems.   
We will call an
iterated function system, $\{f_n\}_{n\in\N}$ Gauss like if
$$\cup_{i=0}^{\infty} f_i([0,1])=[0,1]$$
and if for all $x\in [0,1]$ we have that $f_i(x)<f_j(x)$ implies
$i>j$. We then have that
\begin{thm}\label{Gauss}
If $\{f_i\}_{i=1}^{\infty}$ is a Gauss like system, $\alpha>1$ and
$\Phi(n)= n^{\alpha}$ then
$$\dim_H X_{\Phi}=\frac{1}{1+\alpha(d-1)}.$$
\end{thm}
Previous work on this type of problem has been done in the case of continued fractions. Here the maps $f_n:[0,1]\rightarrow\ [0,1]$ can be defined by $f_n(x)=\frac{1}{x+n}$ for each $n\in\N$. In 1941 Good showed that the set where
$\lim_{i\rightarrow\infty}a_i=\infty$ has dimension $\frac{1}{2}$
(\cite{Good}) and this was extended by Ramharter, \cite{Ram}, to show that the set of $x$ with strictly increasing continued fraction exponents has dimension $\frac{1}{2}$. We will show that this dimension is unchanged if we
use the stronger condition $a_{i+1}>\beta a_i$ for $\beta>1$ and for all $i\in\N$. However on the other hand we will show that if we have the condition $a_{i+1}(x)>(a_{i}(x))^{\alpha}$ for all $i\in\N$ and $\alpha>1$ then the dimension does drop below $\frac{1}{2}$. 
Subsequent to the work of Good several papers, e.g \cite{L} and
\cite{WW}, have added conditions on the rate of convergence of the
$a_i$ to infinity either along sequences or subsequences. In particular \cite{WW} calculate the Hausdorff dimension of the set where $a_i(x)\geq\Phi(x)$ for infinitely many $n$ for any function $\Phi$. 

In this setting of Continued fractions Theorem \ref{thm:main} and Theorem \ref{Gauss} have the following corollary:
\begin{cor}\label{cf}
If we denote the continued fraction expansion of $x$ by $a_1(x),a_2(x),a_3(x),...$ then we have that
\begin{enumerate}
\item
for any $\beta\geq 1$ we have that:
$$\dim_H\{x:a_{i+1}(x)\geq\beta a_{i}(x)\text{ for all }i\in\N\}=\frac{1}{2};$$
\item
for any $\alpha>1$ we have that
$$\dim_H\{x:a_{i+1}(x)\geq (a_{i}(x))^{\alpha}\text{ for all }i\in\N\}=\frac{1}{1+\alpha}.$$
\end{enumerate}
\end{cor}
It should be noted that in part 1. of the Corollary the case where $\beta=1$ was shown by Ramharter in \cite{Ram}. The second part of this Corollary  relates to the work by \L uczak in \cite{L}. Here for $\alpha,\beta>1$ the sets
$$\Theta[\alpha,\beta]=\{x:a_n(x)\geq \beta^{\alpha^n}\text{ for all }n\in\N\}$$
are considered (where $a_i(x)$ denote the continued fraction digits of $x$). It is shown that $\dim\Theta[\alpha,\beta]=\frac{1}{1+\alpha}$ which corresponds with the dimension found in Part 2. of Corollary \ref{cf}. This connection is no surprise since if we have that $a_{i+1}> a_i^{\alpha}$ for all $i\in \N$ then $a_n(x)> a_1(x)^{\alpha^n}$.

Finally we can show that Theorem \ref{Gauss} does not hold if we consider more general systems. In particular if there are gaps between the first level cylinders
then the situation can be significantly different as illustrated by the
following theorem:
\begin{thm}\label{existence}
For any $d>1$ and any strictly increasing function
$\Phi:\N\rightarrow\N$ there exists a $d$-decaying system,
$\{f_i\}_{i=1}^{\infty}$ such that
$$\dim_H X_{\Phi}=\frac{1}{d}.$$
\end{thm}

Throughout the paper if $x\in\Lambda$ we will denote the $n$th
level cylinder containing $x$ by
$$C_n(x):=f_{a_1(x)}\circ\cdots\circ f_{a_n(x)}([0,1]).$$
 The rest of the paper is laid
out as follows. In section 2 we prove some lemmas which are the
key to the proofs of our main theorems. Theorem \ref{thm:main} and
Theorem \ref{packing} are then proved in section 3. Finally section
4 and 5 are devoted to the proofs of Theorems \ref{Gauss} and
\ref{existence} respectively.

We would like to thank Omri Sarig and Marc Kesseb\"{o}hmer for subsequent useful discussions. In the particular we would like to thank Marc Kesseb\"{o}hmer for informing us of the work of Remharter of which we were previously unaware.

\section{Key Lemmas}

In this section we prove a series of lemmas with will be used in
the following sections. We start with two lemmas needed to prove
the upper bound for Theorems \ref{thm:main} and \ref{packing}. We
fix $d>1$ and a $d$-decaying iterated function system
$\{f_i\}_{i=1}^{\infty}$. For a positive integer $k$ we will use
$\Lambda_k$ to denote the attractor of the system
$\{f_i\}_{i=k}^{\infty}$.

\begin{lem}\label{subsyst}
$$\lim_{k\rightarrow\infty}\dim_H\Lambda_k=\frac{1}{d}$$
\end{lem}
\begin{proof}
In the case of the Gauss map this result is contained in the work of Good, \cite{Good} and the precise asymptotic for the rate of convergence is given in \cite{JK}. For more general systems it will follow from Bowen's formula for the Hausdorff dimension of infinite iterated function systems given in \cite{MU}. However some of the systems we are considering do not
satisfy the assumptions in \cite{MU} and so we include a proof.
First of all we prove that $\dim_H\Lambda_k\geq\frac{1}{d}$. We
fix any $s<\frac{1}{d}$. We can then find $m\in\N$ such that
$\sum_{i=k+1}^m\xi_i^s\geq 1$. If we consider the iterated
function system consisting of the maps $f_k,\ldots,f_m$ and let
$\Lambda_k,m\subset\Lambda$ be the attractor. By standard results
for iterated function systems $\dim_H\Lambda_{m,k}\geq s_m$ where
$s_m$ is the solution to $\sum_{i=k+1}^m\xi_i^{s_m}=1$ and we know
by definition that $s_m\geq s$. Since this holds for any
$s<\frac{1}{d}$ we know that $\dim_H\Lambda\geq\frac{1}{d}$.

To obtain the upper bound  we fix $s>\frac{1}{d}$ and choose $k$
such that $\sum_{i=k}^{\infty}\lambda_i^s\leq 1$. For convenience
we will denote $\N(k)$ to be the set of natural numbers greater
than or equal to $k$. We get that
\begin{eqnarray*}
\sum_{(a_1,\ldots,a_n)\in \N(k)^n}|a_1,\ldots,a_n|^{s}&\leq&\sum_{(a_1,\ldots,a_n)\in \N(k)^n}(\lambda_{a_1}\cdots\lambda_{a_n})^s\\
&\leq&\left(\sum_{i=k}^{\infty}\lambda_i^s\right)^n\leq 1.
\end{eqnarray*}
It then follows that $\dim_H\Lambda_k\leq s$. Thus for every
$s>\frac{1}{d}$ we can find $k$ such that $\dim_H\Lambda_k\leq s$
since $\dim_H\Lambda_k$ is clearly monotonically decreasing the
result follows.
\end{proof}
We also need analogue of Lemma \ref{subsyst} in terms of upper box
dimension or equivalently packing dimension. We will let
$$s_0=\overline{\dim}_B(\{f_i(0)\}_{i=1}^{\infty}.$$
Note that instead of $0$ we could take any other point of interval
$[0,1]$ and the value of $s_0$ would not change.

\begin{lem}\label{subsystpack}
$$\lim_{k\to\infty}\dimp\Lambda_k\leq\max\left\{\frac{1}{d},s_0\right\}$$
\end{lem}
\begin{proof}
Let $\epsilon>0$, $\delta << \epsilon$ and let $K$ be sufficiently
large such that
$\sum_{i=K}^{\infty}(C_2(\delta))^{d^{-1}+\epsilon}i^{-(d-\delta)(d^{-1}+\epsilon)}
\leq 1$. We will fix $K^{-d}<\lambda<1$. We can find a constant
$N_0>0$ such that for any integer $n\geq 0$ we can cover
$\{f_j(v)\}_{j=1}^{\infty}$ by $N_0\lambda^{-n(s_0+\epsilon)}$
intervals of size $\frac {C_1(\delta)}
{C_2(\delta)}\lambda^{(1+2\delta/d)n}$.

We let $\N(K)^*$ denote all finite words formed from the alphabet
$\N(K)$. Let

$$A_n=\{\omega\in\N(K)^*:\lambda^{n}>|f_{\omega}([0,1])|\geq\lambda^{n+1}\}.$$
We have that $\#A_n\leq \lambda^{-(n+1)(d^{-1}+\epsilon)}$. We now
fix integer $N>0$ and find a cover of $\Lambda_K$ with intervals
of length $\lambda^N$. Let $0<n\leq N$ and let $\omega\in A_n$. We
denote

$$D(\omega)=\bigcup\{f_{\omega}\circ f_j([0,1]); |f_{\omega}\circ f_j([0,1])|\leq\Lambda^N\}$$
where $j\in\N$. We have that

$$\Lambda_k\subset \cup_{n=0}^{N}\cup_{\omega\in A_n}D(\omega).$$
For $\omega\in A_n$ we know that $\{f_j(0)\}_{j=1}^{\infty}$ can
be covered by at most $N_0\lambda^{-N(s_0+\epsilon)}$ intervals of
size $\frac {C_1(\delta)} {C_2(\delta)}\lambda^{(1+2\delta/d)n-N}$
and so $D(\omega)$ can be covered by
$N_1\lambda^{(n-N)(s_0+\epsilon)}$ intervals of size
$\lambda^{-N}$ for a constant $0<N_1\leq 3N_0$. Therefore we have
that $\Lambda_K$ can be covered by

$$N_1\sum_{n=0}^N \lambda^{-(n+1)(d^{-1}+\epsilon)-(N-n)(s_0+\epsilon)}$$
intervals of size $\lambda^N$. Thus

$$\overline{\dim}_B\Lambda_k\leq \limsup_{n\rightarrow\infty}\frac{\log\sum_{n=0}^N \lambda^{-n(d^{-1}+\epsilon)-(N-n)(s_0+\epsilon)}}{-N\log\lambda}\leq\max\{d^{-1},s_0\} + \epsilon.$$
Applying Theorem 3.1 in \cite{MU} completes the proof.
\end{proof}
We now let $\Phi:\N\rightarrow\R$ satisfy $\Phi(n)\geq n$ for all $n\in\N$ and let the set $X_{\Phi}$ be as defined in (\ref{def1}).
To prove the lower bounds in Theorems \ref{thm:main} and
\ref{packing} we introduce certain subsets of $X_{\Phi}$ which we
will use in order to define a measure supported on $X_{\Phi}$. For
any natural $n$ let $l(n)$ be the minimal natural number such that

\begin{equation} \label{eqn:ln}
\sum_{i=[\Phi(n)]+1}^{l(n)} \xi_i ^{1/d-\epsilon} \geq 1
\end{equation}
where $[\Phi(n)]$ denotes the integer part of $\Phi(n)$.
We will then let $K$ be the smallest integer such that for any $k\geq K$ we have

$$k^{-d-\epsilon}\leq\xi_k\leq\lambda_k\leq k^{-d+\epsilon}.$$
We then define $\{l_n\}_{n\in\N}$ recursively by $l_1=K$ and
$l_{n+1}=l(l_n)$. Let $Y_{\Phi, \epsilon}$ be a subset of
$X_\Phi$ defined as
\[
Y_{\Phi, \epsilon} = \{x\in [0,1]; \Phi(l_n) <a_n(x)\leq
l_{n+1}\}.
\]

\begin{lem} \label{lem:ln}
There exist $\gamma > 1$ such that
\[
\frac {l_{n+1}} {\Phi(l_n)} < \gamma
\]
for all $n$.
\end{lem}
\begin{proof}
By assumption we have that $C_1 i^{-d-\varepsilon}\leq\xi_i\leq
C_2 i^{-d+\varepsilon}$. Thus we have that if $n$ is sufficiently
large

$$2\geq\sum_{i=\Phi(n)+1}^{l(n)}\xi_i^{1/d-\epsilon}\geq C_2\int_{\Phi(n)+1}^{l(n)} s^{(-d+\epsilon)(1/d-\epsilon)}\text{d}s.$$
Evaluating this integral and using the fact that $2(d+\frac 1 d-\epsilon)\epsilon<1$ gives that
$$l(n)^{(d+\frac 1 d-\epsilon)\epsilon}\leq 2\Phi(n)^{(d+\frac 1 d-\epsilon)\epsilon}$$
and the result easily follows.
\end{proof}

The following lemma is the key to the lower bound for Theorems
\ref{thm:main}, \ref{packing} and \ref{existence}.

\begin{lem}\label{measure}
We can define a probability measure $\nu$ supported on $Y_{\Phi,\epsilon}$ such that
\begin{enumerate}
\item\label{p1}
$\nu(\mathcal{C}_n(x))<|\mathcal{C}_n(x)|^{1/d-\epsilon}$ for all
$x\in Y_{\Phi,\epsilon}$. \item\label{p2} For $\nu$ almost all
$x\in Y_{\Phi,\epsilon}$
$$\limsup_{r\rightarrow 0}\frac{\log(\nu(B(x,r)))}{\log r}\geq \frac{1}{d}-\epsilon.$$
\end{enumerate}
\end{lem}
\begin{proof}
We start by fixing a positive integer $n$ considering the set of
integers $I(n):=[\Phi(L(n))+1,l(n+1)]$. We will then refine this
set by removing the integers which refer to the left most and
right most intervals. To be precise let

$$I'(n)=\{i\in I_n:\exists j,k\in I_n\text{ with } f_j([0,1])\leq f_i([0,1])\leq f_k([0,1])\}$$
(where $J_1\leq J_2$, is to be understood as: interval $J_1$ is to
the left of the interval $J_2$). We denote by $s_n$ the value such
that

$$\sum_{i\in I'(n)}\xi_i^{s_n}=1$$
and note that each $s_n\geq\frac{1}{d}-\epsilon$. For each $n$ we
will define a finite measure $\mu_n$ supported on the finite
sigma-algebra given by the sets $\{f_i([0,1])\}_{i\in I'(n)}$ and
satisfying that $\mu_n(\mathcal{C}_{\omega_i})=\xi_i^{s_n}$. We
can then let $\nu_n=\otimes_{k=0}^{n-1}\mu_k\circ T$ and note that
the extension $\nu$ of these measures (Kolmogorov) will be
supported on a subset of $Y_{\beta,\epsilon}$.

Note that for any cylinder $\mathcal{C}_{\omega_1\ldots \omega_n}$ we have that

$$\nu(\mathcal{C}_{\omega_1\ldots \omega_n})=\xi_{\omega_1}^{s_1}\cdots\xi_{\omega_n}^{s_n}$$
and we can immediately deduce \ref{p1}.

For \ref{p2} let $x\in\text{supp}(\mu)$ and fix an $n$. We can
then deduce that $a_{n+1}(x)\in I'(n)$. Now consider the set of
cylinders

$$Z_n=\{\Pi([a_1(x),\ldots,a_n(x),j]\}_{j\in I(n)}$$
and let $R_n=\min_{j\in I(n)}\xi_{a_1(x)}\cdots\xi_{a_n(x)}\xi_j$.
We know that $x\in [a_1(x),\ldots,a_n(x),j]$ for some $j\in I'(n)$
therefore $B(x,R_n)\subset C_n(x)$ and $B(x,R_n)$ will intersect
at most two members of $Z_n$. Therefore we have that

\begin{eqnarray*}
\mu(B(x,R_n)&\leq& 2C_2^{s_{n+1}}\xi_{a_1(x)}^{s_1}\cdots\xi_{a_n(x)}^{s_n}(\Phi(l_n)+1)^{s_{n+1}(-d+\epsilon)}\\
&\leq&2\gamma^{-s_{n+1}(-d-\epsilon)}l_{n+1}^{2\epsilon}
C_2^{s_{n+1}}(\xi_{a_1(x)}\cdots\xi_{a_n(x)})^{1/d-\epsilon}(l_{n+1})^{s_{n+1}(-d-\epsilon)}.
\end{eqnarray*}
Thus if we take logarithms we have that
$$\log\mu(B(x,R_n))\leq (1/d-\epsilon)\log R_n+2\epsilon\log l_{n+1}+\mathit{o}(\log R_n)$$
and to complete the proof we notice that $\log l_{n+1}/\log R_n$
is uniformly bounded.
\end{proof}

\section{Proofs of Theorems \ref{thm:main} and \ref{packing}} \label{2}
\subsection*{Proof of Theorem \ref{thm:main}}

We fix $d>1$, a $d$-decaying system $\{f_i\}_{i=1}^{\infty}$ and a
  function $\Phi:\N\rightarrow\R$ satisfying
$n\leq \Phi(n)\leq \beta n$ for all $n\in\N$ and some $\beta\geq 1$. To
prove the upper bound we note that for any $k\in\N$
$$X_{\Phi}\subset\bigcup_{l\leq k} \bigcup_{a_1<\ldots<a_l\leq k} f_{a_l}\circ\ldots\circ f_{a_1}(\Lambda_k).$$
Since the maps $f_i$ are bi-Lipschitz, it then follows by Lemma
\ref{subsyst} that $\dim_H X_{\Phi}\leq\frac{1}{d}$.

To compute the lower bound for any $x\in X_{\Phi}$ and $n\in\N$ we
let $r_n(x):=|\mathcal{C}_n(x)|$. We can freely assume that
$\beta$ is strictly greater than 1 (If $\Psi\geq \Phi$ then $X_{\Psi}\subset X_{\Phi}$). We then have the following
result
\begin{lem} \label{lem:rn}
For any $\delta>0$ there exist $l>0$ and $N>0$ such that for any
$x,y\in Y_{\Phi,\epsilon}$ and $n>N$ we have
\[
r_n(x) > \left(r_{n+l}(y)\right)^{1+\delta}.
\]
\end{lem}
\begin{proof}
   By applying Lemma \ref{lem:ln} we can calculate that for any $l\in\N$
\begin{eqnarray*}
\frac {r_n(x)} {(r_{n+l}(y))^{1+\delta}} &=& \frac {r_n(x)} {(r_n(y))^{1+\delta}}  \cdot \left( \frac
{r_n(y)} {r_{n+l}(y)}\right)^{1+\delta}\\
&\geq&\left(\frac{C_1(\delta/2d)}{(C_2(\delta/2d)\gamma)^{1+\delta}}\right)^n \cdot \frac{1}{(C_1(\delta/2d))^{l(1+\delta)}} \beta^{ndl(1+\delta)}.
\end{eqnarray*}
Thus if we choose $l$ large enough such that
   $$\beta^{ld(1+\delta)}>\frac{1}{(C_1(\delta/2d))^{l(1+\delta)}}\left(\frac{C_1(\delta/2d)}{(C_2(\delta/2d) \gamma)^{1+\delta}}
\right)$$
then the proof is complete.
\end{proof}
Hence, for any $x\in Y_{\Phi,\varepsilon}$, $n>N$, and
$(r_{n+l+1}(x))^{1+\delta}\leq r\leq (r_{n+l}(x))^{1+\delta}$, the set $B_r(x)\cap
Y_{\Phi,\varepsilon}$ will be contained in $\mathcal{C}_n(x)\cup \mathcal{C}_n(y)$
for some $y\in Y_{\Phi, \varepsilon}$. Thus we will have

\[
\liminf_{r\rightarrow 0}\frac{\log\nu(B(x,r))}{\log r}\geq \liminf_{n\to \infty} \inf_{y\in
Y_{\beta, \varepsilon}} \frac {(1/d-\epsilon) \log r_n(y) +\log 2}
{(1+\delta) \log r_{n+2l+1}(y)}.
\]
The only thing missing in the proof of Theorem \ref{thm:main} is a
comparison of sizes of $r_n(x)$ and $r_{n+1}(x)$.

\begin{lem}
There exists a sequence $v_n\to 1$ such that for every $x\in
Y_{\beta,\epsilon}$,
\[
\frac {\log r_{n+1}(x)} {\log r_n(x)} <v_n
\]
\end{lem}
\begin{proof}
We have that
\[
r_{n+1}(x) \geq \xi_{a_{n+1(x)}} r_n(x).
\]
Thus it suffices to show that $\frac{\log \xi_{a_{n+1}(x)}}{\log
r_n(x)}$ tends to $0$ uniformly in $x$. For $\epsilon>0$ we have
that for all $x$

$$r_n(x)\leq\prod_{i=1}^n \frac{C_2}{i^{d-\epsilon}}\leq \frac{C_2^n}{(n!)^{d-\epsilon}}.$$
On the other hand by Lemma  and the definition of $\Phi$

$$l_n+1\leq\gamma\Phi(l_n)\leq\beta\gamma l_n.$$
Thus $a_{n+1}(x)\leq \beta\gamma^{n}$  and so
$\xi_{a_{n+1}(x)}\geq (\beta\gamma)^{-n(d+\epsilon)}$ and the
result follows.
\end{proof}

\subsection*{Proof of Theorem \ref{packing}}
We fix $d>1$, a $d$-decaying system $\{f_i\}_{i=1}^{\infty}$  and
a function $\Phi:\N\rightarrow \R$ such that $\Phi(n)\geq n$. We will let
$s_0=\overline{\dim}_B(\{f_i(0)\}_{i=1}^{\infty})$. To show that
$\dimp\leq\min\{1/d,s_0\}$ we simply replicate the upper bound in
the proof of Theorem \ref{thm:main} with Lemma \ref{subsystpack}
replacing Lemma \ref{subsyst}. The fact that $\dimp X_{\Phi}\leq
\frac{1}{d}$ can immediately be deduced from Lemma
\ref{measure}.

We now turn to the case where $s_0\geq\frac{1}{d}$. First we need
to show that the upper box counting dimension and the packing
dimension of $X_{\Phi}$ are the same.

\begin{lem}\label{boxeqpacking}
We have that for any function $\Phi:\N\to\N$ with $\Phi(n)\geq n$
$$\dimp X_\Phi=\overline{\dim}_B X_{\Phi}.$$
\end{lem}
\begin{proof}
It can easily be seen that the proof of Theorem 3.1 in \cite{MU} can be applied in this situation.
\end{proof}

We let $J$ denote the closure of $X_{\Phi}$ and note that by Lemma
\ref{boxeqpacking} we can deduce that $\dimp
X_\Phi=\overline{\dim}_B J$. We will let $v$ be some accumulation
point of $\{f_i(1)\}$. We then have that
$J\supset\{f_i(v)\}_{i=1}^{\infty}$ and so $\overline{\dim}_B
J\geq s_0$ and the result immediately follows by Lemma
\ref{boxeqpacking}.

\section{Proof of Theorem \ref{Gauss}}

We fix $d>1$, a Gauss like $d$-decaying system
$\{f_i\}_{i=1}^{\infty}$, $\alpha>1$ and a function
$\Phi:\N\rightarrow \R$ such that $\Phi(n)= n^{\alpha}$. Denote
$s=1/(1+\alpha (d-1))$. It is enough to prove that for every $K>1$
\[
\dim_H X_{\Phi,K} =\frac 1 {1+\alpha (d-1)},
\]
where
\[
X_{\Phi, K} = \{x\in X_\Phi; a_1(x)=K\}.
\]
Indeed, we have
\[
X_{\Phi, 2} \subset X_\Phi = \{x_0\} \cup \bigcup_{n=0}^\infty
\bigcup_{K=2}^\infty f_1^n X_{\Phi, K},
\]
where $x_0$ is the fixed point of $f_1$. We fix $K>1$, $\delta>0$
and denote $C_1=C_1(\delta), C_2=C_2(\delta)$.

Given $x\in X_\Phi$ we define
\[
\Delta_n(x) = \bigcup \{\mathcal{C}_{n+1}(y); y\in
\mathcal{C}_n(x)\cap X_\Phi\}.
\]
Obviously, it is the union of all $(n+1)$-st level subcylinders of
$\mathcal{C}_n(x)$, where the $(n+1)$-st coordinate is at least
$a_n(x)^\alpha$. We have
\[
C_1^n \prod_{i=1}^n a_i(x)^{-d-\delta} \leq |\mathcal{C}_n(x)| \leq C_2^n
\prod_{i=1}^n a_i(x)^{-d+\delta}
\]

and
\begin{equation} \label{eqdelta}
C_1^{n+1} C_4^{-1} a_n(x)^{-(d+\delta-1)\alpha} \prod_{i=1}^n a_i(x)^{-d-\delta} \leq
|\Delta_n(x)| \leq C_2^{n+1} C_4 a_n(x)^{-(d-\delta-1)\alpha} \prod_{i=1}^n
a_i(x)^{-d+\delta}.
\end{equation}
We will distribute on $X_{\Phi, K}$ a probabilistic measure
$\mu$, satisfying $\mu(a_1(x)=K) =1$ and

\begin{equation}\label{lzor}
          \mu(a_{n+1}(x)=j|a_n(x)=i) =
              \begin{cases}
              0& \text{ if } j<i^\alpha,\\
              c_i i^{\alpha (d-1) s} j^{-(d+\alpha (d-1))s}  &\text{ if } j\geq i^\alpha,
              \end{cases}
\end{equation}
where
\[
c_i = \frac 1 {\sum_{j\geq i^\alpha} i^{\alpha (d-1) s}
j^{-(d+\alpha (d-1))s}}
\]
is a normalising constant. It is easy to check that for some $C_3>1$ we
have
\[
C_3^{-1} \leq c_i \leq C_3
\]
for all $i$ (in fact, $c_i\to (d+\alpha (d-1))s +1$ as $i\to
\infty$).

The reason we have chosen the measure $\mu$ in this way is that  for all $x\in X_{\Phi, K}$ we have for each $n$

\begin{eqnarray*}
&&C_3^{-n} \prod_{i=2}^n a_i(x)^{-ds} \cdot a_1(x)^{\alpha (d-1)s} a_n(x)^{-\alpha (d-1) s} \leq  \mu(\Delta_n(x)) = \mu(\mathcal{C}_n(x))\\
&&\leq C_3^n \prod_{i=2}^n a_i(x)^{-ds} \cdot a_1(x)^{\alpha (d-1)s} a_n(x)^{-\alpha (d-1) s}
\end{eqnarray*}


Comparing this with \eqref{eqdelta} we have that for all $x\in X_{\Phi, K}$
\begin{equation} \label{delta}
C_5^{-n} |\Delta_n(x)|^{(1+c\delta)s} \leq
\mu(\Delta_n(x)) \leq C_5^n |\Delta_n(x)|^{(1-c\delta)s}.
\end{equation}

Note that

\begin{equation} \label{n2}
|\Delta_n(x)| < |\mathcal{C}_n(x)| \leq \prod_{i=1}^n C_2 K^{-(d-\delta)\alpha^{i-1}}
= C_2^n K^{-(d-\delta) (\alpha^n-1)/(\alpha -1)}.
\end{equation}
Hence for $x\in X_{\Phi,K}$ we can calculate
\begin{eqnarray*}
\frac{\log\mu(B(x,|\Delta_n(x)|))}{\log|\Delta_n(x)|}&\leq&\frac{\log\mu(\Delta_N(x))}{\log |\Delta_n(x)|}\\
&\leq& s(1+c\delta)+\frac{\mathit{o}(\log|\Delta_n(x)|)}{\log |\Delta_n(x)|}.
\end{eqnarray*}
Thus we can conclude that
\[
\dim_H X_{\Phi,K} \leq s(1+c\delta).
\]

For the lower bound on the Hausdorff dimension we will use
Frostman Lemma, again. Denote
\[
r_n(x)=|\Delta_n(x)|, \ R_n(x)=|\mathcal{C}_n(x)|.
\]

We already know that
\[
\lim_{n\to \infty} \frac {\log \mu(\Delta_n(x))} {\log
|\Delta_n(x)|} \geq s(1-c\delta).
\]
$B_{r_n(x)}(x)$ contains $\Delta_n(x)$ and might intersect at most
one other $\Delta_n(y)$. Moreover, this $\Delta_n(y)$ must be a
neighbouring one, which means that $a_i(x)=a_i(y)$ for $i<n$ and
$|a_n(x)-a_n(y)|=1$. Hence, by \eqref{delta} we have that
$$\mu(B_{r_n(x)}(x)) \leq (2+\varepsilon) C_5^n r_n(x)^{(1-c\delta)s}.$$
We then have that
\[
\log\mu(B_{r_n(x)}(x))\leq (1-c\delta)s\log r_n + \mathit{o}(\log r_n).
\]

We need to use this estimate to find $\frac{\log\mu(B_r(x))}{\log
r}$ for $r_n(x)<r<R_n(x)$ and $R_{n+1}(x)<r<r_n(x)$. The first of
these ranges is easy: each $\mathcal{C}_n(x)\setminus \Delta_n(x)$ has
length comparable to $|\mathcal{C}_n(x)|$. Hence, the ball $B_r(x)$ for
$r_n(x)<r<R_n(x)$ will be much bigger than $B_{r_n(x)}(x)$ but
will still intersect at most $\Delta_n(x)$ plus one more
$\Delta_n(y)$. So, in this range
\[
\frac {\log \mu(B_r(x))} {\log r} \geq (1-c\delta)s-\mathit{o}(1).
\]

In the range $R_{n+1}(x)<r<r_{n}(x)$ the ball $B_r(x)$ will
actually intersect several $\mathcal{C}_{n+1}(y), y\in X_{\Phi,
K}$. Let us define

\[
D_r(x) = \bigcup \{\mathcal{C}_{n+1}(y); y\in X_{\Phi, K} \cap B_r(x)\}.
\]
Note that $\mu(D_r(x)) \geq \mu(B_r(x))$ but $|D_r(x)| \leq 2
C_2^{n+1}/C_1^{n+1} r^{1-\delta}$. Hence, we can use $D_r(x)$ instead of
$B_r(x)$ to estimate the local dimension of $\mu$ at $x$ and the
estimation will change at most by a factor $(1\pm \delta)$.

The set $D=D_r(x)$ is an union of consecutive $n+1$-st level
cylinders $\mathcal{C}_{n+1}(y)$ with $a_i(y)=a_i(x)$ for $i\leq
n$ and $l_1\leq a_{n+1}(y) \leq l_2$, where $l_1\geq
a_{n}(x)^\alpha$ and $l_2\leq \infty$. We have
$\mathcal{C}_n(x)=\bigcup_{i=l_1}^{l_2} \mathcal{C}_{n+1}(y_i)$
(where $y_i$ is a point from $\mathcal{C}_n(x)\cap X_{\Phi,K}$
with $n+1$-st symbol in the symbolic expansion equal to $i$. We
have

\[
|\mathcal{C}_{n+1}(y_i)| \geq i^{-d} |\mathcal{C}_n(x)|^{1+c\delta},
\]

hence

\[
|D| \geq |\mathcal{C}_n(x)|^{1+c\delta} \sum_{i=l_1}^{l_2} i^{-d} \approx
(l_1^{-(d-1)} - l_2^{-(d-1)}) |\mathcal{C}_n(x)|^{1+c\delta}.
\]

We also have

\[
|\Delta_{n+1}(y_i)| \leq |\mathcal{C}_n(x)|^{1-c\delta} i^{-d-\alpha(d-1)},
\]
hence by \eqref{delta}

$$\mu(D) = \sum_{i=l_1}^{l_2} \mu(\Delta_{n+1}(y_i)) \leq C_5^{n+1} |\mathcal{C}_n(x)|^{(1-2c\delta)s}
\sum_{i=l_1}^{l_2} i^{-(d+\alpha(d-1))s(1-c\delta)}.$$

Note that

\[
\sum_{i=l_1}^{l_2} i^{-(d+\alpha(d-1))s(1-c\delta)} \leq l_2^{(d+\alpha(d-1))sc\delta} \cdot
\sum_{i=l_1}^{l_2} i^{-(d+\alpha(d-1))s} \leq |D|^{-c\alpha\delta} \cdot \sum_{i=l_1}^{l_2} i^{-(d+\alpha(d-1))s}.
\]


Thus we have that

\begin{eqnarray*}
\log \mu(D) &\leq& \log \left(|\mathcal{C}_n(x)|^{(1-2c\delta)s} \sum_{i=l_1}^{l_2} i^{-(d+\alpha(d-1))s}\right)- c\alpha\delta\log |D| +\mathit{o}(\log(|D|))\\
&\approx& \log((l_1^{-(d+\alpha(d-1))s+1} - l_2^{-(d+\alpha(d-1))s+1})|\mathcal{C}_n(x)|^{(1-2c\delta)s})-c\alpha\delta \log |D|+\mathit{o}(\log(|D|))\\
&=& \log((l_1^{-(d-1)s} - l_2^{-(d-1)s})
|\mathcal{C}_n(x)|^s)+\mathit{o}(\log(|D|))
\end{eqnarray*}

where we use that

\[
(d+\alpha(d-1))s-1=(d-1)s.
\]
By the concavity of function $x\to x^s$ for $s<1$, we have that
\[
a=b^s \wedge\ c=d^s \implies (a-c)\leq (b-d)^s.
\]
Hence we can conclude that
\[
\log(\mu(D))\leq s(1-(3+\alpha)c\delta)\log|D|+\mathit{o(\log |D|)}
\]
and the proof is complete. \qed

\section{Proof of Theorem \ref{existence}}

We start by fixing an increasing function
$\Phi:\N\to\N$ and $d>1$. We need to find a $d$-decaying system $\{f_i\}_{i=1}^{\infty}$ such
that
\[
\dim_H X_\Phi = \frac 1 d.
\]
We will fix $\varepsilon >0$. As in section \ref{2}, we define by $l(n)$
the smallest number for which
\[
\sum_{i=\Phi(n)+1}^{l(n)} C^{1/d - \varepsilon} i^{-1+
d\varepsilon} \geq 1.
\]
We define $l_1=1$ and $l_{n+1}=l(l_n)$. As in Lemma
\ref{lem:ln},
we have that
\[
l_{n+1} < \gamma \Phi(l_n)
\]
for some $\gamma >1$.

The system will be piecewise linear of the form, $T_i(x)=\frac C {i^d} x +
a_i$. We will have that
$$C=\frac{1}{\sum_{i=1}^{\infty}i^{-d}+\sum_{n=1}^{\infty}n^{-2}l_{n+1}^{-1}(l_{n+1}-\Phi(l_n))}.$$
We define the constants $a_i$ recursively by letting $a_1=1-ci^{-d}$ and let

$$
a_n=\left\{\begin{array}{lll}a_{n-1}-Cn^{-d}&\text{ if }&n\notin (\Phi(l_n),l_{n+1})\text{ for any }n\in\N\\
a_{n-1}-Cn^{-d}-Cj^{-2}l_{j+1}^{-1}&\text{ if }&n\in
(\Phi(l_j),l_{j+1})\text{ for some }j\in\N\end{array}\right..
$$

As in section \ref{2} and Lemma \ref{measure} we can define

\begin{equation} \label{pogis}
\tilde{X}_\Phi = \{x: \Phi(l-{n-1})+1 \leq a_n(x) \leq l_n(x)\}
\end{equation}

and distribute on $\tilde{X}_\Phi$ a measure $\nu$ such that

\[
\nu(\mathcal{C}_n(x)) \leq |\mathcal{C}_n(x)|^{(1/d - \varepsilon)}
\]
for all $x\in \tilde{X}_\Phi$. For $x\in \tilde{X}_{\Phi}$ let
$Z_n(x)$ denote the minimal interval containing
$\mathcal{C}_n(x)\cap \tilde{X}_\Phi$. We can calculate

\begin{eqnarray*}
|Z_n(x)|&\approx& |\mathcal{C}_n(x)|\left(Cn^{-2}l_{n+1}^{-1}(l_{n+1}-\Phi(l_n))+\sum_{i=\Phi(l_n)}^{l_{n+1}}i^{-d}\right)\\
&\approx& n^{-2} |\mathcal{C}_n(x)|
\end{eqnarray*}
We can calculate that for any cylinder $\mathcal{C}_n(x)$, $i\neq
j\in (\Phi(l_n),l_n+1]$ that the cylinders
$\mathcal{C}_{n+1}(y_i)$ and $\mathcal{C}_{n+1}(y_j)$ will be
separated by a gap of length at least $Cn^{-2}l_{n+1}^{-1}$.

Hence, for $r_{n+1}<r\leq r_n$

\[
\mu(B_r(x)) \leq g_n(r)=\left( 1+c r \frac {n^2
(l_{n+1}-\Phi(l_n))} {r_n}\right) r_{n+1}^{1/d-\varepsilon}
\]

(where $r_n=|\mathcal{C}_n(x)|$). Note that

\[
g_n(r)\leq cr^{1/d-\varepsilon}
\]
for $r=r_n$ and for $r=r_{n+1}$, and $g_n(r)$ is a linear function
in-between. As $x\to x^{1/d}$ is a concave function, we have

\[
g_n(r) < cr^{1/d-\varepsilon}
\]
for $r_{n+1}<r< r_n$. Hence,

\[
\liminf_{r\to 0} \frac {\log \mu(B_r(x))} {\log r} \geq \frac 1 d
-\varepsilon
\]
and the proof is complete.
\qed

\end{document}